\documentclass[11pt]{amsart}
\pdfoutput=1

\usepackage[dvipsnames, table, xcdraw]{xcolor}
\usepackage{upgreek}
\usepackage{amssymb,amsmath,mathtools,amsfonts,amsthm}
\usepackage[mathscr]{eucal}
\usepackage{fullpage}
\usepackage{color}
\usepackage{caption}
\usepackage{subcaption}
\usepackage{enumitem}
\setlist[itemize]{leftmargin=30pt, itemsep=2pt}
\setlist[enumerate]{leftmargin=30pt, itemsep=2pt}
\usepackage{soul}
\usepackage{setspace} 
\usepackage[numbers]{natbib}
\usepackage{indentfirst} 

\usepackage[colorlinks,pagebackref]{hyperref}

\definecolor{mylinkcolor}{RGB}{0,0,255}
\definecolor{mycitecolor}{RGB}{169,169,169}
\definecolor{myurlcolor}{RGB}{255,20,147}
\definecolor{mybiburlcolor}{RGB}{80,80,80}
\hypersetup{
  colorlinks=true,
  urlcolor=myurlcolor,
  citecolor=mycitecolor,
  linkcolor=mylinkcolor,
  linktoc=page,
  breaklinks=true
}
\usepackage{tikz-cd,tikz} 
\usepackage{cleveref} 

\newtheorem{theorem}{Theorem}
\numberwithin{theorem}{section}
\numberwithin{equation}{section}
\newtheorem{proposition}[theorem]{Proposition}
\newtheorem{coro}[theorem]{Corollary}  
\newtheorem{lemma}[theorem]{Lemma}
\newtheorem{conjecture}[theorem]{Conjecture}
\theoremstyle{definition}\newtheorem{defn}[theorem]{Definition}
\theoremstyle{definition}\newtheorem{remark}[theorem]{Remark}  
\theoremstyle{definition}\newtheorem{question}[theorem]{Question}
\theoremstyle{definition}
\crefname{proposition}{Proposition}{Propositions}
\crefname{coro}{Corollary}{Corollaries}
\crefname{lemma}{Lemma}{Lemmas}
\crefname{conjecture}{Conjecture}{Conjectures}
\crefname{theorem}{Theorem}{Theorems}
\crefname{question}{Question}{Questions}
\crefname{defn}{Definition}{Definitions}
\crefname{remark}{Remark}{Remarks}
\crefname{equation}{equation}{equations}
\Crefname{equation}{Equation}{Equations}

\newcommand{\RR}{\mathbb{R}} 
\newcommand{\ZZ}{\mathbb{Z}}
\newcommand{\QQ}{\mathbb{Q}} 
\newcommand{\CC}{\mathbb{C}} 
\newcommand{\PP}{\mathbb{P}} 
 
\newcommand{\cE}{\mathcal{E}}
\newcommand{\cF}{\mathcal{F}} 
\newcommand{\cM}{\mathcal{M}} 
\newcommand{\cB}{\mathcal{B}} 
\newcommand{\cA}{\mathcal{A}}
\newcommand{\GG}{\mathbb{G}} 
\newcommand{\cO}{\mathcal{O}}

\newcommand{\cH}{\mathcal{H}} 
\newcommand{\ove}{\overline{e}} 
\newcommand{\ovf}{\overline{f}} 
\newcommand{\ovd}{\overline{d}} 

\newcommand{\overbar}[1]{\mkern 1.5mu\overline{\mkern-1.5mu#1\mkern-1.5mu}\mkern 1.5mu}

\DeclareMathOperator{\codim}{codim} 
\DeclareMathOperator{\rank}{rk} 
\DeclareMathOperator{\Spec}{Spec}
\DeclareMathOperator{\Proj}{Proj} 
\DeclareMathOperator{\Sym}{Sym} 
\DeclareMathOperator{\Tsch}{Tsch} 
\DeclareMathOperator{\Hom}{Hom} 
\DeclareMathOperator{\Bun}{Bun} 
\DeclareMathOperator{\polytope}{\mathcal{P}} 
\DeclareMathOperator{\qpolytope}{\mathcal{Q}} 
\DeclareMathOperator{\ev}{ev} 
\DeclareMathOperator{\End}{End}

\DeclareMathOperator{\pr}{pr}
\DeclareMathOperator{\GL}{GL}
\DeclareMathOperator{\SL}{SL}

\renewcommand{\subset}{\subseteq}
\newcommand{\mymat}[4]{\ensuremath{\left( \begin{smallmatrix} #1 & #2 \\ #3 & #4 \end{smallmatrix} \right)}}
\newcommand{\marksing}{\mathcal{U}}

\newcommand{\I}{I}
\newcommand{\J}{J}
\newcommand{\K}{K}
\renewcommand{\L}{L}

\title[Tschirnhausen Bundles of Quintic Covers of $\PP^1$]{Tschirnhausen Bundles of Quintic Covers of $\PP^1$}
\date{October 9, 2025}

\author{Sam Frengley} 
\email{samuel.frengley@inria.fr}
\urladdr{\url{https://samfrengley.github.io/}}

\author{Sameera Vemulapalli}
\address{Department of Mathematics,
  Harvard University}
\email{vemulapalli@math.harvard.edu}
\urladdr{\url{https://web.math.princeton.edu/~sameerav/}}

\usepackage{longtable}
\allowdisplaybreaks

\begin{document}

\begin{abstract}
  A degree $d$ genus $g$ cover of the complex projective line by a smooth irreducible curve $C$ yields a vector bundle on the projective line by pushforward of the structure sheaf. We classify the bundles that arise this way when $d = 5$. Equivalently, we classify which $\PP^3$-bundles over $\PP^1$ contain smooth irreducible degree $5$ covers of $\PP^1$.

  Our main contribution is proving the existence of \emph{smooth} covers whose structure sheaf has the desired pushforward. We do this by showing that the substack of singular curves has positive codimension in the moduli stack of finite flat covers with desired pushforward. To compute the dimension of the space of singular curves, we prove a (relative) ``minimization theorem'', which is the geometric analogue of Bhargava's sieving argument when computing the densities of discriminants of quintic number fields.
\end{abstract}

\maketitle

\setcounter{tocdepth}{1}
\tableofcontents

\section{Introduction}
Let $\pi \colon C \rightarrow \PP^1$ be a degree $d$ cover of the complex projective line by a smooth irreducible curve of genus $g$. The restriction of functions gives a map $\cO_{\PP^1} \rightarrow \pi_* \cO_C$.  Define the \emph{Tschirnhausen bundle of $\pi$} to be $\mathcal{E}_{\pi} \coloneqq (\pi_* \cO_C /\cO_{\PP^1})^{\vee}$. Vector bundles on $\PP^1$ split completely, so we may write $\cE_\pi \cong \cO(e_1)\oplus \cdots \oplus \cO(e_{d-1})$ with $e_1 \leq e_2 \leq \dots \leq e_{d-1}$. Because $h^0(\PP^1, \pi_* \cO_C) = h^0(C,\cO_C) = 1$, we have $e_1 \geq 1$. Furthermore, from $h^1(C,\cO_C) = g$, we have $\sum_i e_i = g + d - 1$. The integer $e_i$ is often referred to as the \emph{$i^{\text{th}}$ scrollar invariant} of $\pi$. For the rest of this article, assume $1 \leq e_1 \leq \dots \leq e_{d-1}$.

\begin{question}[Tschirnhausen realization problem]
\label{q:main}
    Which $(d-1)$-tuples $(e_1,\dots,e_{d-1}) \in \ZZ^{d-1}$ arise as the scrollar invariants of a smooth irreducible degree $d$ cover $C \rightarrow \PP^1$?
\end{question}

Casnati and Ekedahl~\cite[Theorem 1.3]{CE} proved that $C$ admits a relative canonical embedding:
\[\begin{tikzcd}
	C && \PP(\cE) \\
	& \PP^1
	\arrow[hook, from=1-1, to=1-3]
	\arrow["\pi"', from=1-1, to=2-2]
	\arrow["\tau", from=1-3, to=2-2]
\end{tikzcd}
\]
where $\PP(\cE) = \Proj_{\PP^1} (\Sym^\bullet \cE)$.

Conversely, any such embedding of a curve $C$ into a $\PP^{d-2}$-bundle over $\PP^1$ must arise this way. Thus \Cref{q:main} is equivalent to the following.

\begin{question}
\label{q:main-alt}
Which $\PP^{d-2}$-bundles over $\PP^1$ contain smooth irreducible covers of genus $g$ and degree $d$?
\end{question}

For $d \in \{2,3,4\}$, the answer to the Tschirnhausen realization problem is known completely. The degree $2$ case is trivial and $e_1 = g+1$. In degree $3$, the pair $(e_1,e_2)$ arises as the scrollar invariants of a smooth irreducible curve if and only if $e_2 \leq 2e_1$, as shown by Maroni~\cite[Section~3]{maroni} (see also \cite[Section~2.2]{VV}). In degree $4$, it is shown in \cite[Theorem~3.1]{VV} that a triple $(e_1,e_2,e_3)$ arises as the scrollar invariants of a smooth irreducible curve if and only if $e_3 \leq e_1 + e_2$. 

For general $d$, the second named author formulated the following precise conjecture which classifies which Tschirnhausen bundles arise from \emph{primitive} covers, i.e., covers which do not factor through a nontrivial proper subcover. 

Define the polytope 
\begin{equation}\label{eq:V}
\polytope_d := \Big\{ (\overline{e}_1,\dots,\overline{e}_{d-1}) \;  : \; \sum \overline{e}_k = 1,  \;  0 \leq \overline{e}_1 \leq \cdots \leq \overline{e}_{d-1},
\text{ and }\overline{e}_{k + \ell} \leq \overline{e}_k + \overline{e}_\ell \Big\} \subset \RR^{d-1}.
\end{equation}

\begin{conjecture}[{\cite[Conjecture~1.3]{VV}}]
\label{conj:tschirnhausen}
For each $d$ and $g$, the tuple $(e_1,\dots,e_{d-1})$ is achieved as the scrollar invariants of a smooth primitive cover $C \to \PP^1$ of degree $d$ if and only if
\[
    \frac{1}{g + d-1}(e_1,\dots,e_{d-1}) \in \polytope_d.
\]
\end{conjecture}

\Cref{conj:tschirnhausen} is true when $d = 3$ by work of Maroni~\cite[Section~3]{maroni} and when $d = 4$ by \cite[Theorem~3.1]{VV}. Additionally, \cite[Theorem~1.4]{VV} shows that the scrollar invariants of any smooth primitive cover lie in $\polytope_d$. It remains to show the converse: that every tuple of integers which lies in $\polytope_d$ (when appropriately scaled) is achieved as the scrollar invariants of a smooth primitive curve. 

In this article, we prove \Cref{conj:tschirnhausen} when $d = 5$. Because all degree $5$ covers are primitive, we also answer the Tschirnhausen realization problem in this case. 

\begin{theorem}
\label{thm:main-thm}
Fix an integer $g \geq 0$. The tuple $(e_1,\dots,e_{4})$ is achieved as the scrollar invariants of a smooth irreducible genus $g$ cover $C \to \PP^1$ of degree $5$ if and only if
\[
    \frac{1}{g + 4}(e_1,\dots,e_{4}) \in \polytope_5.
\]
\end{theorem}

In fact, we deduce \Cref{thm:main-thm} from the following more precise result. Recall there is an embedding $C \xhookrightarrow{} \PP(\cE)$. In the case $d = 5$, the minimal resolution of $C$ inside $\PP(\cE)$ has the form
\[
    0 \rightarrow \tau^*\det(\cE)(-5)  \rightarrow \tau^*(\cF^{\vee}\otimes \det(\cE))(-3) \rightarrow \tau^*\cF(-2) \rightarrow \mathcal{I}_C \rightarrow 0
\]
where $\tau \colon \PP(\cE) \to \PP^1$ is the structure morphism, $\cF$ is a vector bundle on $\PP^1$ known as the \emph{first syzygy bundle of $\pi$}, and $\mathcal{I}_C$ is the ideal sheaf of $C$ in $\PP(\cE)$. The vector bundle $\cF$ has degree $2(g+4)$ and rank $5$. As with $\cE$, we write $\cF \cong \cO(f_1) \oplus \dots \cO(f_5)$ where $f_1 \leq f_2 \leq \dots \leq f_5$. For simplicity, set $d^{(k)}_{ij} \coloneqq f_i + f_j + e_k - (g+4)$ for all $1 \leq i < j \leq 5$ and $1 \leq k \leq 4$.

\begin{theorem}
\label{thm:deg-5} 
Fix an integer $g \geq 0$ and let $\cE$ and $\cF$ be rank $4$ and $5$ vector bundles on $\PP^1$ with $\deg(\cE) = g+4$ and $\deg(\cF) = 2(g+4)$. There exists a smooth irreducible quintic cover $\pi \colon C \rightarrow \PP^1$ with Tschirnhausen bundle $\cE$ and first syzygy bundle $\cF$ if and only if:
\begin{enumerate}[label=(\roman*)]
    \item \label{cond1} for all $1 \leq i < j \leq 5$ and $1 \leq k \leq 4$ with $i + j + k = 8$, we have $d^{(k)}_{ij} \geq 0$;
    \item \label{cond2} and if $d^{(1)}_{15} < 0$ and $d^{(1)}_{24} < 0$ and $d^{(4)}_{12} < 0$, then $d^{(1)}_{25} = 0$.
\end{enumerate} 
\end{theorem}

If $f_1 + f_2 + e_1 \geq g+4$, then the three conditions above are satisfied. In this case, Casnati proved a Bertini-like theorem which shows the existence of a smooth curve with prescribed $\cE$ and $\cF$ bundle \cite[Theorem~1.2]{casnati} (cf. \cite{P}). Conversely, Canning and Larson \cite[Section~5]{CL} show that every smooth irreducible quintic cover satisfies the linear inequalities stated in \ref{cond1}. Our contribution is twofold: first, we show the necessity of \ref{cond2}, and second, when all \emph{ three} conditions are satisfied, we show the existence of a smooth irreducible cover with prescribed $\cE$ and $\cF$ bundle. To deduce \Cref{thm:main-thm} from \Cref{thm:deg-5}, we simply project the linear inequalities in \Cref{thm:deg-5} onto the $e_i$. 

We conclude with one final corollary, which follows immediately from \Cref{thm:deg-5} and Casnati's parametrization of quintic covers with prescribed $\cE$ and $\cF$ bundle. Let $\cH_{\cE,\cF}$ be the moduli stack of smooth irreducible quintic covers with prescribed $\cE$ and $\cF$ bundle, and let $\cH_{5,g}$ be the moduli stack of smooth irreducible quintic covers of genus $g$. Our result shows that $\cH_{\cE,\cF}$ is not always of ``expected codimension'' $h^1(\End(\cE)) + h^1(\End(\cF))$.

\begin{coro}
\label{coro:dim}
Suppose $\cE$ and $\cF$ satisfy the conditions of \Cref{thm:deg-5}. Then $\cH_{\cE,\cF}$ is irreducible and the codimension of $\cH_{\cE,\cF}$ in $\cH_{5,g}$ is $h^1(\End(\cE)) + h^1(\End(\cF)) - h^1(\bigwedge^2 \cF \otimes \cE \otimes \det \cE^{\vee})$.
\end{coro}

\subsection{Proof overview}
We make use of the parametrization of quintic covers of $\PP^1$ given by Casnati \cite[Theorem~1.1]{casnati}. In particular, the moduli space of quintic covers with fixed $\cE$ and $\cF$ bundle can be presented as a quotient of a (possibly empty) Zariski open subset of the affine space $H^0( \bigwedge^2 \cF \otimes \cE \otimes \det \cE^{\vee})$. Our contribution is to use sieve techniques pioneered by Bhargava in the arithmetic setting (see e.g., \cite{B-quinticdiscs}), which translate to partial normalization in the geometric setting, to prove that a general member of $H^0(\bigwedge^2 \cF \otimes \cE \otimes \det \cE^{\vee})$ corresponds to a smooth quintic cover in the required cases.

In \Cref{sec:necessity}, we use Casnati's parametrization to show the necessity of condition \ref{cond2} in \Cref{thm:deg-5} (condition \ref{cond1} is known to be necessary by the computations in \cite[Section~5]{CL}). To finish the proof of \Cref{thm:deg-5}, it then remains to show the \emph{existence} of smooth covers with the required $\cE$ and $\cF$ bundle.

In \Cref{sec:minimization}, we prove \Cref{prop:move-stuff}, which is the key technical ingredient in the proof of \Cref{thm:deg-5}. Given a section $\cA \in H^0( \bigwedge^2 \cF \otimes \cE \otimes \det \cE^{\vee})$ which does \emph{not} correspond to a smooth quintic cover, we prove a (relative) ``minimization theorem'' which shows that there exists a $\GL(\cE) \times \GL(\cF)$-translate of $\cA$ which satisfies certain congruence conditions at the singular point. We further remark that by work of Bhargava~\cite{B-hcl4}, each section $\cA$ actually gives rise to a finite flat degree $5$ morphism $\pi \colon X \rightarrow \PP^1$, where $X$ need not be smooth, irreducible, or even Gorenstein. Our (relative) minimization theorem gives rise to a partial normalization $\widetilde{\pi} \colon \widetilde{X} \to \PP^1$ for which the corresponding Tschirnhausen and first syzygy bundles are given in \Cref{lemma:normalize-quintic}.

In \Cref{sec:dim-counting}, we use the congruence conditions provided by \Cref{prop:move-stuff} to compute the codimension in $H^0( \bigwedge^2 \cF \otimes \cE \otimes \det \cE^{\vee})$ of the space of sections which do not correspond to a smooth quintic cover. We show that this codimension is positive in the cases allowed by \Cref{thm:deg-5}, thus concluding the proof of \Cref{thm:deg-5}. We deduce \Cref{thm:main-thm} from \Cref{thm:deg-5} in \Cref{sec:finishing-proof}.

Finally, in \Cref{sec:discussion}, we address several conjectures posed in \cite[Section~8]{VV} in the case $d = 5$, using our main result. We also discuss the mysterious condition~\ref{cond2} in \Cref{thm:deg-5}.

\subsection{Conventions}
We work over $\CC$ for simplicity, but we note here that the proofs of \Cref{thm:main-thm} and \Cref{thm:deg-5} generalize to any field of sufficiently large characteristic (where sufficiently large depends only on the genus $g$). Henceforth, the phrase \emph{cover of $\PP^1$} means a finite flat morphism $C \rightarrow \PP^1$. The phrase \emph{smooth cover of $\PP^1$} means further that $C$ is smooth. Likewise, the phrase \emph{irreducible cover of $\PP^1$} means that $C$ is irreducible. 

Throughout $d$ will refer to the degree of $\pi$ and $g$ will refer to the arithmetic genus of $C$.  Throughout, $\cE$ will be a vector bundle on $\PP^1$ of rank $4$ and degree $g+4$ and we write $\cE \simeq \cO(e_1) \oplus \dots \oplus \cO(e_4)$ with $0 < e_1 \leq \dots \leq e_4$. Similarly, $\cF$ will be a vector bundle of rank $5$ and degree $2(g+4)$ and we write $\cF \simeq \cO(f_1) \oplus \dots \oplus \cO(f_5)$ with $f_1 \leq \dots \leq f_5$.

\subsection{Acknowledgements}
The authors would like to thank Ravi Vakil for useful discussions and Uthsav Chitra for suggesting the proof of \Cref{lem:projection}. SF was funded by the Royal Society through C{\'e}line Maistret's Dorothy Hodgkin Fellowship. SV was supported by the National Science Foundation under grant number DMS2303211. 

\section{The parametrization of quintic covers of \texorpdfstring{$\PP^1$}{P{\textasciicircum}1}}
\label{sec:param-algebras}
Let $\cH_{\cE, \cF}$ denote the stack parametrizing smooth quintic covers $\pi \colon C \to \PP^1$ with Tschirnhausen bundle $\cE$ and first syzygy bundle $\cF$. Here $\deg(\cF) = 2(g+4)$, $\rank(\cF) = 5$, $\deg(\cE) = g+4$, and $\rank(\cE) = 4$. It turns out that $\cH_{\cE, \cF}$ can naturally be described as the quotient of (an open subscheme of) an affine space by a group, as we now show. 

Upon choosing splittings $\cE \cong \cO(e_1) \oplus \cdots \oplus \cO(e_4)$ and $\cF \cong \cO(f_1) \oplus \cdots \oplus \cO(f_5)$, a section $\cA \in H^0( \bigwedge\nolimits^2 \cF \otimes \cE \otimes \det \cE^{\vee} )$ can be considered as a quadruple $\cA = (A_1, A_2, A_3, A_4)$ of $5 \times 5$ alternating matrices such that the $(i,j)$-th entry of $A_k$ is a homogeneous form $a_{ij}^{(k)} \in \CC[s,t]$ whose degree is given by $d_{ij}^{(k)} = \deg(a_{ij}^{(k)}) = f_i + f_j + e_k - (g+4)$. Since the matrices $A_k$ are alternating, $a_{ij}^{(k)} = -a_{ji}^{(k)}$.

Given a section $\cA \in H^0( \bigwedge\nolimits^2 \cF \otimes \cE \otimes \det \cE^{\vee} )$, consider the $5\times 5$ alternating matrix of linear forms $\cA(\mathbf{x}) = \sum_{k=1}^4 A_k x_k$ in the variables $x_1,\dots,x_4$. Consider the scheme $C \subset \PP(\cE)$ cut out by the five $4 \times 4$ principal sub-Pfaffians of the matrix $\cA(\mathbf{x})$ (which are quadratic forms in the variables $x_k$). Let $U_{\cE,\cF}$ be the (open) locus of sections cutting out a smooth curve. 

Let $G \subset \GL(\cE) \times \GL(\cF)$ be the subgroup consisting of pairs of matrices $(g_4, g_5)$ for which $\det(g_4)^2 = \det(g_5)$. The space $H^0( \bigwedge\nolimits^2 \cF \otimes \cE \otimes \det \cE^{\vee} )$ has a natural action of $G$; given a section $\cA \in H^0( \bigwedge\nolimits^2 \cF \otimes \cE \otimes \det \cE^{\vee} )$ and $g = (g_4, g_5) \in G \subset \GL(\cE) \times \GL(\cF)$ the matrix $g_4 \in \GL(\cE)$ acts linearly on the span of the matrices $A_1, A_2, A_3, A_4$ as 
\[
\begin{bmatrix}
    A_1 & A_2 & A_3 & A_4  
\end{bmatrix}^t
\mapsto 
\frac{1}{\det(g_4)} g_4 
\begin{bmatrix}
  A_1 & A_2 & A_3 & A_4
\end{bmatrix}^t
\]
and the element $g_5 \in \GL(\cF)$ acts via 
\[
A_k \mapsto g_5 \, A_k \, g_5^t. 
\]

The utility of this $G$-action is captured in the following theorem.

\begin{theorem}[{Casnati~\cite[Theorem~1.1]{casnati}}]
    \label{prop:moduli-qunitic}
    There exists an isomorphism of moduli stacks
    \[
    \cH_{\cE, \cF} \cong U_{\cE,\cF}/ G.
    \]
\end{theorem}

The primary contribution of this article is to prove that $\cE$ and $\cF$ satisfy the hypotheses of \Cref{thm:deg-5} precisely when the subvariety $U_{\cE,\cF}$ is nonempty.

\section{The necessity of condition \texorpdfstring{\textnormal{\ref{cond2}} in \Cref{thm:deg-5}}{(ii) in Theorem 1.5}}
\label{sec:necessity}
Given a section $\cA \in H^0( \bigwedge\nolimits^2 \cF \otimes \cE \otimes \det \cE^{\vee} )$, let $C$ denote the scheme $C \subset \PP(\cE)$ cut out by the five $4 \times 4$ sub-Pfaffians of the matrix $\cA(\mathbf{x})$. Note that $C$ may not be a curve: for example if $\cA$ is the zero section, then $C$ is equal to $\PP(\cE)$. 

If any of the linear inequalities in \Cref{thm:deg-5}\ref{cond1} are not met, then the computations in \cite[Section~5]{CL} imply that $C$ is reducible or $\dim C > 1$ (see also \cite[Lemma~10]{B-quinticdiscs} for analogous computations in the arithmetic case). In this section we show the necessity of \Cref{thm:deg-5}\ref{cond2}.

\begin{defn}
\label{smooth-defn}
Given a section $\cA$ and a point $P \in C$, we say $\cA$ is \emph{smooth at $P$} if the Zariski tangent space of $C$ at $P$ is $1$-dimensional. If $\cA$ is smooth at all points of $C$, we say $\cA$ is \emph{smooth}. Given a point $p \in \PP^1$, we say $\cA$ is \emph{smooth above $p$} if it is smooth at all points $P \in C$ above $p$. If $\cA$ is not smooth at $P$, we say $\cA$ is \emph{singular at $P$}. Similarly, if $\cA$ is not smooth, we say it is \emph{singular}. Finally, if $\cA$ is not smooth above $p$, we say $\cA$ is \emph{singular above $p$}.
\end{defn}

Observe that $\cA$ is smooth at a point $p$ if and only if $C$ is a smooth curve. 

\begin{lemma}
\label{lem:singular-shape}
Let $p \in \PP^1$ be a point with uniformizer $u$. Fix $1 \leq \I < \J \leq 5$ and $1 \leq \K \leq 4$. Suppose that:
\begin{enumerate}[label=(\roman*)]
    \item $a^{(\K)}_{\I\J} \equiv 0 \pmod{u^2}$;
    \item $a^{(k)}_{\I\J} \equiv 0 \pmod{u}$ for all $1 \leq k \leq 4$; 
    \item $a^{(\K)}_{\I j} \equiv 0 \pmod{u}$ for all $1 \leq j \leq 5$; and
    \item $a^{(\K)}_{i \J} \equiv 0 \pmod{u}$ for all $1 \leq i \leq 5$;
\end{enumerate}
Then $\cA$ is singular above $p$.
\end{lemma}

\begin{proof}
  Let $R = \CC[u]$. Since singularity may be checked locally we may pass to the affine open subset $V = \Spec R \subset \PP^1$ containing $p$. Consider the closed subvariety $D \subset \PP^3 \times V$ cut out by the vanishing of the $4 \times 4$ sub-Pfaffians of $\cA \in R^4 \otimes \bigwedge^2 R^5$, i.e. let $D$ be the restriction of $C$ to $\PP^3 \times V$. On applying an appropriate permutation matrix $g_4 \in \GL_4(\CC)$ we may assume that $\K = 4$. After applying an appropriate permutation matrix $g_5 \in \GL_5(\CC)$, we may assume that $\I = 4$ and $\J = 5$.
  
  Since $A_4$ has rank at most $2$ at $u = 0$, the point $P = ([0:0:0:1], 0)$ lies on $D$. A direct calculation shows that the Jacobian matrix of the $4 \times 4$ sub-Pfaffians of $\cA(\mathbf{x})$ evaluated at $P$ has rank at most $2$, so the dimension of the Zariski tangent space of $C$ at $P$ is at least $2$.
\end{proof}

\begin{proposition}
\label{prop:cond2}
Suppose $d^{(1)}_{15}, d^{(1)}_{24}, d^{(4)}_{12} < 0$. Then $\cA$ is singular above every point on $\PP^1$ at which $a^{(1)}_{25}$ vanishes.
\end{proposition}
\begin{proof}
If $u$ is a uniformizer at a point at which $a^{(1)}_{25}$ vanishes, then the matrix $\cA$ is in the shape given in \Cref{lem:singular-shape}.
\end{proof}

\Cref{prop:cond2} implies that if $d^{(1)}_{15}, d^{(1)}_{24}, d^{(4)}_{12} < 0$, then for $C$ to be a smooth curve it is necessary that $d^{(1)}_{25} = 0$, as claimed in \Cref{thm:deg-5}. 

\section{Minimization for quintic covers of \texorpdfstring{$\PP^1$}{P{\textasciicircum}1}}
\label{sec:minimization}

\subsection{Normal forms and the minimization theorem}
The main contribution of this section is to prove \Cref{prop:move-stuff} which states that singular degree $5$ covers $C \to \PP^1$ arise from sections $\cA \in H^0( \bigwedge\nolimits^2 \cF \otimes \cE \otimes \det \cE^{\vee} )$ of a very specific form, which we now describe.

\begin{defn}
\label{def:normal-form}
Let $\cA \in H^0( \bigwedge\nolimits^2 \cF \otimes \cE \otimes \det \cE^{\vee} )$ be a section. Let $p \in \PP^1$ be a point and let $u$ be a uniformizer at $p$. Let $1 \leq \I < \J \leq 5$ and $1 \leq \K \leq 4$. 
\begin{enumerate}[label=(\Alph*)]
    \item 
    \label{def:k-normal-form}
    We say $\cA$ is in \emph{normal form of type $(\K)$} at $p$ if $A_{\K} \equiv 0 \pmod{u}$.
    
    \item \label{def:ijk-normal-form}
    We say $\cA$ is in \emph{normal form of type $(\I,\J,\K)$} at $p$ if:
    \begin{enumerate}[label=(\roman*)]
        \item \label{nf:2} $a^{(k)}_{\I\J} \equiv 0 \pmod{u}$ for all $1 \leq k \leq 4$; 
        \item \label{nf:3} $a^{(\K)}_{\I j} \equiv 0 \pmod{u}$ for all $1 \leq j \leq 5$; 
        \item \label{nf:4} $a^{(\K)}_{i\J} \equiv 0 \pmod{u}$ for all $1 \leq i \leq 5$; and
        \item \label{nf:1} $a^{(\K)}_{\I\J} \equiv 0 \pmod{u^2}$.
    \end{enumerate}
\end{enumerate}
\end{defn}

\begin{remark}
By a direct calculation one may verify that if $\cA$ is in either of the normal forms in \Cref{def:normal-form} then $\cA$ is singular at the point above $u = 0$ where all coordinates in $\PP^3$ are $0$ except the $\K$-th coordinate. See \Cref{lem:singular-shape} for the proof in the case of the normal form $(\I,\J,\K)$. 

In particular, if $\cA$ is in a normal form of type $(\I,\J,\K)$, then the resulting scheme has (at best) a \emph{nodal} singularity. If $\cA$ is in a normal form of type $(\K)$ then the singular point is a quintuple point. We have not computed the singularity type. 
\end{remark}

The central goal of this section is to prove the following result which puts a singular section $\cA$ into a normal form.

\begin{proposition}
  \label{prop:move-stuff}
  Suppose that a section $\cA \in H^0(\bigwedge\nolimits^2 \cF \otimes \cE \otimes \det \cE^\vee )$ is singular above a point $p \in \PP^1$. Then there exists $g \in G$ such that if $\cB = g(\cA)$, then either:
  \begin{enumerate}[label=(\Alph*)]
      \item there exists an integer $1 \leq \K \leq 4$ such that $\cB$ is in normal form of type $(\K)$; or
      \item there exists a triple $(\I, \J, \K)$ with $1 \leq \I < \J \leq 5$ and $1 \leq \K \leq 4$ such that $\cB$ is in normal form of type $(\I, \J, \K)$.
  \end{enumerate}
\end{proposition}

\subsection{Aside: partial normalization}
Before proceeding with the proof of \Cref{prop:move-stuff} we record the following corollary. Although this result is not strictly necessary to prove our theorem, it is an interesting observation and was our motivation for proving \Cref{prop:move-stuff}. 

Work of Bhargava~\cite{B-hcl4} gives a generalization of Casnati's construction of smooth quintic covers of $\Spec(\ZZ)$. Namely, given $\cA \in \ZZ^4 \otimes \bigwedge^2 \ZZ^5$, Bhargava produces a finite flat degree $5$ cover $\pi \colon X \rightarrow \Spec(\ZZ)$ (and this construction works even when $\cA$ is the zero section). The cover need not be smooth, irreducible, or even Gorenstein. Bhargava's construction is to give an explicit basis and formulae for the multiplication coefficients of the ring $\pi_* \cO_X$ (cf. \cite{evan} for an extension to Dedekind domains).  

This construction immediately extends to the base scheme $\PP^1$ (in characteristic $0$), simply by working on affine charts. The only additional subtlety is that while vector bundles over $\Spec(\ZZ)$ are all trivial, vector bundles over $\PP^1$ need not be trivial. In particular, every section $\cA \in H^0(\bigwedge\nolimits^2 \cF \otimes \cE \otimes \det \cE^\vee )$ gives rise to a finite flat degree $5$ morphism $\pi \colon X \rightarrow \PP^1$. The scheme $X$ is isomorphic to the scheme $C$ defined by the vanishing of the $4 \times 4$ sub-Pfaffians of $\cA$ precisely when $X$ is Gorenstein. 

By abuse of notation, we call the bundle $\cF$ a first syzygy bundle of $X$, even though $X$ may not necessarily embed into $\PP(\cE)$. If $X$ is not Gorenstein, it may even have multiple first syzygy bundles. The scheme $X$ is singular precisely when the section $\cA$ is singular.

\begin{coro}
\label{lemma:normalize-quintic}
    If $\cA$ is a singular section giving rise to a finite flat cover $\pi \colon X \rightarrow \PP^1$, then there exists a birational morphism $\widetilde{X} \to X$ for which the following diagram commutes:
\[
\begin{tikzcd}
	\widetilde{X} && X \\
	& \PP^1
	\arrow[from=1-1, to=1-3]
	\arrow[from=1-1, to=2-2]
	\arrow[from=1-3, to=2-2]
\end{tikzcd}
\]
Moreover $\widetilde{X}$ has Tschirnhausen bundle $\widetilde{\cE} \cong \bigoplus_{k=1}^4 \cO(\widetilde{e}_k)$ and a first syzygy bundle $\widetilde{\cF} \cong \bigoplus_{i=1}^5 \cO(\widetilde{f}_i)$ where either:
\begin{enumerate}[label=(\Alph*)]
    \item there exists an integer $1 \leq \K \leq 4$ such that 
    \begin{equation*}
        \widetilde{e}_k = \begin{cases}
            e_k - 2 & \text{if $k = \K$, and} \\
            e_k - 1 & \text{otherwise,}
        \end{cases}
        \qquad \text{and} \qquad
        \widetilde{f}_i = f_i - 2 ,
    \end{equation*}
    or
    \item there exists a triple $(\I, \J, \K)$ such that 
    \begin{equation*}
        \widetilde{e}_k = \begin{cases}
            e_k-1  & \text{if $k = \K$, and} \\
            e_k & \text{otherwise,}
        \end{cases}
        \qquad \text{and} \qquad
        \widetilde{f}_i = \begin{cases}
            f_i - 1  & \text{if $i = \I$ or $\J$, and} \\
            f_i & \text{otherwise.}
        \end{cases}        
    \end{equation*}
\end{enumerate}
\end{coro}
\begin{proof}
Let $u \in \CC[s,t]$ be a homogeneous linear form vanishing at the image of a singular point of $X$ under the morphism $\pi$. By \Cref{prop:move-stuff} we may assume without loss of generality that $\cA$ is in a normal form of type $(\I,\J,\K)$ or type $(\K)$. In the former case, we have: 
\begin{enumerate}[label=(\roman*)]
    \item $a_{\I\J}^{(k)} = u \xi_{\I\J}^{(k)}$ for some $\xi_{\I\J}^{(k)} \in \CC[s,t]$ for each $k \neq \K$;
    \item $a_{i\J}^{(\K)} = u \xi_{i\J}^{(\K)}$ for some $\xi_{i \J}^{(\K)} \in \CC[s,t]$ for each $i \neq \J$;
    \item $a_{j\I}^{(\K)} = u \xi_{j\J}^{(\K)}$ for some $\xi_{j \I}^{(\K)} \in \CC[s,t]$ for each $j \neq \I$; and
    \item $a_{\I\J}^{(\K)} = u^2 \xi_{\I\J}^{(\K)}$ for some $\xi_{i \J}^{(\K)} \in \CC[s,t]$.
\end{enumerate}

Consider $g = (g_4, g_5) \in \GL_4(\CC(\PP^1)) \times \GL_5(\CC(\PP^1))$ to be the pair of diagonal matrices where $g_4$ has $u^{-1}$ in the $(\K, \K)$ entry and $1$ elsewhere and $g_5$ has $u^{-1}$ in the $(\I, \I)$ and $(\J, \J)$ entries and $1$ elsewhere. Consider $\widetilde{\cA} \in H^0(\bigwedge\nolimits^2 \widetilde{\cF} \otimes \widetilde{\cE} \otimes \det \widetilde{\cE}^\vee )$ given by $\widetilde{\cA} = g(\cA)$. By construction, the entries of the matrices $\widetilde{A}_k$ are regular forms on $\PP^1$. Let $\widetilde{\pi} : \widetilde{X} \rightarrow \PP^1$ be the cover associated to $\widetilde{A}$. We immediately see that $\widetilde{X}$ has the desired Tschirnhausen and first syzygy bundles. 

We now show the existence of a birational morphism $\widetilde{X} \rightarrow X$ factoring $\widetilde{\pi}$. We have already chosen a splitting of $\cE$ and $\cF$ which induces splittings of $\widetilde{\cE}$ and $\widetilde{\cF}$. Upon fixing these splittings, Bhargava's construction canonically gives splittings of the $\cO_{\PP^1}$-algebras $\pi_* \cO_X$ and $\widetilde{\pi}_* \cO_X$.

With respect to these splittings, the $5 \times 5$ diagonal matrix
\begin{equation}
\label{eqn:matrix}
    \begin{bmatrix}
        1 & 0 \\
        0 & g_{4}^{-1}
    \end{bmatrix}.
\end{equation}
induces an inclusion of sheaves
\begin{equation}
\label{eqn:inclusion}
\iota: \pi_* \cO_X \xhookrightarrow{} \widetilde{\pi}_* \cO_{\widetilde{X}}.
\end{equation}
In other words, $\iota \in \Hom(\pi_* \cO_X, \widetilde{\pi}_* \cO_{\widetilde{X}})$ is given by the matrix in \eqref{eqn:matrix}.

Note that $\deg(g_4)^2 = \det(g_5)$ and $\widetilde{\cA} = g(\cA)$. Since $g_4^{-1}$ and $g_5^{-1}$ are matrices whose entries are are regular forms on $\PP^1$, the functoriality of Bhargava's construction implies that $\iota$ is an inclusion of $\cO_{\PP^1}$-algebras. Therefore $\iota$ induces a morphism of schemes over $\PP^1$:
\[
\begin{tikzcd}
	\widetilde{X} = \underline{\Spec}_{\PP^1}(\widetilde{\pi}_* \cO_{\widetilde{X}}) && \underline{\Spec}_{\PP^1}(\pi_* \cO_X) = X \\
	& \PP^1
	\arrow[from=1-1, to=1-3]
	\arrow[from=1-1, to=2-2]
	\arrow[from=1-3, to=2-2]
\end{tikzcd}
\]
To see that the morphism $\widetilde{X} \rightarrow X$ is birational, simply observe that it is an isomorphism away from the vanishing of $u$.

When $\cA$ is in normal form of type $(\K)$ the result follows from an analogous argument with $g = (g_4, g_5)$ chosen so that $g_4$ is the diagonal matrix with $u^{-2}$ in the $(\K,\K)$ entry and $u^{-1}$ elsewhere and $g_5$ is the scalar matrix $u^{-2} I$. 
\end{proof}

\subsection{Proof of \Cref{prop:move-stuff}}

The remainder of this section is devoted to proving \Cref{prop:move-stuff}. The main technical inputs are \Cref{lemma:normal-form-5-CC,lemma:divisibility} which state that the congruence conditions on the entries of a section $\cA$ in normal form can be obtained locally (in the ring $\CC[u]/u^2$) and may be achieved after only applying unipotent transformations.

Let $U_n \subset \GL_n$ be the subgroup consisting of lower triangular unipotent matrices, i.e., lower triangular matrices with $1$s on the diagonal. We now prove that the conditions of \Cref{prop:move-stuff} can be achieved modulo $u$. A similar argument can be found in \cite[Section~4.2]{B-geomsieve}. 

Recall that if $\cA \in R^4 \otimes \bigwedge^2 R^5$ we write $\cA(\mathbf{x}) = \sum_{k=1}^4 A_i x_i$. We write $\cA(\mathbf{x})_{ij} = \sum_{k=1}^4 a_{ij}^{(k)} x_k$ for the linear form which is the $(i,j)$-entry of the matrix $\cA(\mathbf{x})$. 

\begin{lemma}
    \label{lemma:normal-form-5-CC}
       Let $\overbar{\cA} \in \CC^4 \otimes \bigwedge^2 \CC^5$ be a section such that the scheme $\overbar{C} \subset \PP^3$ cut out by the vanishing of the $4 \times 4$ sub-Pfaffians of $\overbar{\cA}$ has a positive dimensional Zariski tangent space at a point $\overline{p}$. Then there exist unipotent matrices $\overbar{g} \in U_4(\CC) \times U_5(\CC)$, an integer $1 \leq \K \leq 4$, and distinct integers $1 \leq \I,\J,\L \leq 5$ such that if $\overbar \cB = \overbar{g}(\overbar \cA)$ then either:
       \begin{enumerate}[label=(\Alph*)]
           \item \label{norm-form-CC-1} $\overbar{B}_\K = 0$; or
           \item \label{norm-form-CC-2} $\overbar{B}_{k} \neq 0$ for all $k$, and $\bar{b}_{i \I}^{(\K)} = \bar{b}_{i \J}^{(\K)} = \bar{b}_{i \L}^{(\K)} = 0$ for all $1 \leq i \leq 5$, and either:
             \begin{enumerate}[label=(\roman*)]
             \item \label{norm-form-CC-3}
               $\overbar{\mathcal{B}}(\mathbf{x})_{\I\L} = \overbar{\mathcal{B}}(\mathbf{x})_{\J\L} = 0$; or
             \item \label{norm-form-CC-4}
               $\overbar{\mathcal{B}}(\mathbf{x})_{\I\J} = 0$ and the linear forms $\overbar{\mathcal{B}}(\mathbf{x})_{\I\L}$ and $\overbar{\mathcal{B}}(\mathbf{x})_{\J\L}$ are linearly independent.
             \end{enumerate}
        \end{enumerate}
\end{lemma}
\begin{proof}
Let $\K$ be be the largest integer such that the coordinate of $\overbar{p}$ is nonzero there. Then we may apply a transformation in $U_4(\CC)$ and assume without loss of generality that $\overbar{p}$ has all coordinates $0$ except at $\K$, where it has coordinate $1$. Since $\overbar{p}$ lies on $\overbar{C}$, the matrix $\overbar{A}_\K$ has rank $2$.

Let $v = (v_1, ..., v_5)$ and $w = (w_1, ..., w_5)$ be vectors chosen so that $\overbar{A}_{\K} =  v \wedge w$ and consider the matrix
\[
    M = \begin{bmatrix}
        v_1 & v_2 & v_3 & v_4 & v_5 \\
        w_1 & w_2 & w_3 & w_4 & w_5 \\
    \end{bmatrix}.
\]
Observe that $\rank M \leq 1$ if and only if $\overbar{A}_\K = 0$ which is precisely \ref{norm-form-CC-1}. 

From now on assume $\overbar{A}_k \neq 0$ for all $k$ and that $\rank M > 1$. Elements of $U_5(\CC)$ act on $M$ by column operations, so after acting by a matrix $g \in U_5(\CC)$ we may assume that $v_i = w_i = 0$ for all $i \in S \subset \{1, ..., 5\}$ where $|S| = 3$. Let $\{\ell,m\}$ be the complement of $S$.

For each $i,j \in S$ we define the linear forms
\[
\phi_{ij} = \overbar{a}_{\ell m}^{(\K)} \, \overbar{\cA}(\mathbf{x})_{ij}.
\]
A direct calculation shows that the tangent space to $\overbar{C}$ at $\overbar{p}$ is cut out in $\PP^3$ by the linear forms $\phi_{ij}$ for each pair $i,j \in S$. By abuse of notation, write $\overbar{\cA}(\mathbf{x})_{ij}$ for the column vector whose $k^{\text{th}}$ row is the coefficient of $x_k$ in $\overbar{\cA}(\mathbf{x})_{ij}$ and consider the $3 \times 4$ matrix
\[
    N = \begin{bmatrix}
        \overbar{\cA}(\mathbf{x})_{\I\J} & \overbar{\cA}(\mathbf{x})_{\I\L} & \overbar{\cA}(\mathbf{x})_{\J\L}
    \end{bmatrix}
\]
for $S = \{\I,\J,\L\}$. The condition $\rank M > 1$ implies that $\overbar{a}_{\ell m}^{(\K)} \neq 0$, and therefore, since $\overbar{p}$ has a positive-dimensional tangent space, we have $\rank N \leq 2$.

Consider the subgroup $V \subseteq U_5$ whose action on $\CC^5$ fixes the coordinates not contained in $S$. The action of $V$ on the columns of $N$ is given by the action of $V$ on $\bigwedge^2 \CC^3$. Via the isomorphism $\bigwedge^2 \CC^3 \cong \CC^3$, we obtain a map $V \rightarrow \GL_3$ whose image is $U_3$. Therefore, upon ordering the columns of $N$ appropriately, we may make unipotent transformations on the columns of $N$.

If $\rank N \leq 1$ then, upon applying a transformation $h \in U_5(\CC)$, we may assume without loss of generality that two of the columns of $N$ are zero. After possibly re-labelling $\I,\J,\L$  condition \ref{norm-form-CC-3} holds. 
Else if $\rank N = 2$ then, upon applying a transformation $h \in U_5(\CC)$ and re-ordering $\I,\J,\L$, we may assume without loss of generality that $\cA(\mathbf{x})_{\I\J} = 0$ for some pair $\I,\J \in S$. The remaining two columns of $N$ are now linearly independent and \ref{norm-form-CC-4} holds.
\end{proof}

Using \Cref{lemma:normal-form-5-CC} we prove that the conditions of \Cref{def:normal-form} can be achieved modulo $u^2$.

\begin{lemma}
  \label{lemma:divisibility}
  Let $R = \CC[u]/u^2$ and let ${\cA} \in R^4 \otimes \bigwedge^2 R^5$ be such that the scheme cut out in $\PP^3_R$ by the $4 \times 4$ sub-Pfaffians of $\cA$ has Zariski tangent space of dimension at least $2$ at a point ${p}$. Then there exists $g \in U_4(\CC) \times U_5(\CC)$ such that either:
  \begin{enumerate}[label=(\Alph*)]
    \item $g(\cA)$ is in normal form of type $(\K)$ at $p$; or 
      \item $g (\cA)$ is in normal form of type $(\I, \J, \K)$ at $p$.
  \end{enumerate}
\end{lemma} 
\begin{proof}
Because $p$ has Zariski tangent space of dimension $\geq 2$, the restriction of the Zariski tangent space to the closed subvariety $u = 0$ has positive dimension and thus we may apply \Cref{lemma:normal-form-5-CC} to $\overbar{\cA} = \cA \pmod{u}$. Suppose that $\cA$ \emph{cannot} be put into normal form of type $(\K)$ so that $\overbar{\cA}$ must be in the form \ref{norm-form-CC-2} from \Cref{lemma:normal-form-5-CC}. In particular we may assume without loss of generality that there exist distinct integers $1 \leq \I,\J, \L \leq 5$ with $\I < \J$ and an integer $1 \leq \K \leq 4$ such that $a_{i \I}^{(\K)} \equiv a_{i \J}^{(\K)} \equiv a_{i \L}^{(\K)} \equiv 0 \pmod u$ for all $1 \leq i \leq 5$ and either:
\begin{enumerate}[label=(\roman*)]
    \item \label{a-mod-u} $\cA(\mathbf{x})_{\I\L} \equiv \cA(\mathbf{x})_{\J\L} \equiv 0 \pmod{u}$; or
    \item \label{b-mod-u} $\cA(\mathbf{x})_{\I\J}  \equiv 0 \pmod{u}$ and $\cA_{I\L}(\mathbf{x})$ and $\cA_{J\L}(\mathbf{x})$ are linearly independent modulo $u$.
\end{enumerate}

In case \ref{a-mod-u}, if $a_{\I\L}^{(\K)} \equiv 0 \pmod{u^2}$ then $\cA$ is already in a normal form of type $(\I, \L, \K)$. Otherwise without loss of generality suppose $\I < \J$ and let $h \in U_5(\CC)$ be the matrix whose $(\I,\J)$-entry is the value of $a_{\J\L}^{(\K)} / a_{\I\L}^{(\K)}$ at $u = 0$ and whose other non-diagonal entries are $0$. After acting by $h$ we may assume that $a_{\J\L}^{(\K)} \equiv 0 \pmod{u^2}$ in which case $\cA$ is in a normal form of type $(\J, \L, \K)$.

In case \ref{b-mod-u} we claim that the equality $a_{\I\J}^{(\K)} \equiv 0 \pmod{u^2}$ holds without any further transformation. To see this, let $\tau = (\tau_4, \tau_5) \in \SL_4(\CC) \times \SL_5(\CC)$ be a pair of permutation matrices for which $\tau_4$ swaps $\K$ and $4$, and $\tau_5$ swaps $\{\I,\J\}$ for $\{4,5\}$. Writing $\mathcal{B} = \tau(\cA)$, we have $\cB(\mathbf{x})_{45} \equiv 0 \pmod u$, and $b_{ij}^{(4)} \equiv 0 \pmod u$ for all $1 \leq i \leq 5$ and $j = 4,5$. That is, $\tau$ swaps the roles of the matrices $A_\K$ and $A_4$ and swaps the roles of the $\{\I,\J\}$ rows with the $\{4,5\}$ rows. In particular we have $a_{\I\J}^{(\K)} = \pm b_{45}^{(4)}$. Under this transformation, the point $p$ is sent to the point ${q} = ([0:0:0:1],0)$ in the scheme in $\PP^3_R$ cut out by the $4 \times 4$ sub-Pfaffians of $\cB$. 

Now consider a transformation $\gamma = (\gamma_4, \gamma_5) \in \SL_4(\CC) \times \SL_5(\CC)$ of the form
\begin{equation*}
    \gamma_4 = \begin{bmatrix}
      * & * & * & * \\
      * & * & * & * \\
      * & * & * & * \\
      0 & 0 & 0 & 1 
    \end{bmatrix}  
    \qquad
    \gamma_5 = \begin{bmatrix}
        * & * & * & 0 & 0 \\
        * & * & * & 0 & 0 \\
        * & * & * & 0 & 0 \\
        0 & 0 & 0 & 1 & 0 \\
        0 & 0 & 0 & 0 & 1 
    \end{bmatrix}.
\end{equation*}
Let $\mathcal{C} = \gamma(\mathcal{B})$ and note that $b_{45}^{(4)} = \pm c_{45}^{(4)}$ remains unchanged up to sign and $([0:0:0:1], 0)$ is a point of the scheme defined by $\mathcal{C}$ whose tangent space has dimension at least $2$. By choosing $\gamma_5$ appropriately, we may assume that modulo $u$ 
\begin{equation*}    
    C_4 \equiv 
    \begin{bmatrix}
      0 & \rho & 0 & 0 & 0 \\
      -\rho & 0 & 0 & 0 & 0 \\
      0 & 0 & 0 & 0 & 0 \\
      0 & 0 & 0 & 0 & 0 \\
      0 & 0 & 0 & 0 & 0 
    \end{bmatrix}
\end{equation*}
and by choosing $\gamma_4$ we may assume that modulo $u$ we have
\begin{equation*}
  C_1 \equiv \begin{bmatrix}
      0 & 0 & * & * & * \\
      0 & 0 & * & * & * \\
      * & * & 0 & \epsilon & 0 \\
      * & * & -\epsilon & 0 & 0 \\
      * & * & 0 & 0 & 0 
  \end{bmatrix}
  \qquad 
  C_2 \equiv \begin{bmatrix}
      0 & 0 & * & * & * \\
      0 & 0 & * & * & * \\
      * & * & 0 & 0 & \lambda \\
      * & * & 0 & 0 & 0 \\
      * & * & -\lambda & 0 & 0 
  \end{bmatrix}
  \qquad
  C_3 \equiv \begin{bmatrix}
      0 & 0 & * & * & * \\
      0 & 0 & * & * & * \\
      * & * & 0 & 0 & 0 \\
      * & * & 0 & 0 & 0 \\
      * & * & 0 & 0 & 0 
  \end{bmatrix}.
\end{equation*}
where $\rho, \epsilon,\lambda \in \{0,1\}$. Because $\rank A_{\K} = 2$ we may assume $\rho = 1$, and because $\cA_{J\L}(\textbf{x})$ and $\cA_{I\L}(\textbf{x})$ are linearly independent we may assume $\epsilon = \lambda = 1$.

Now, a direct calculation shows that on the affine chart where $x_4 \neq 0$ the Jacobian matrix of the $4 \times 4$ sub-Pfaffians of ${\mathcal{C}}(\mathbf{x})$ evaluated at $([0:0:0:1], 0)$ contains the matrix 
\[
\begin{bmatrix}
    0 & \rho \epsilon & 0 \\
    0 & 0 & \rho \lambda \\
    \rho \tfrac{\partial}{\partial u} c_{45}^{(4)} & \rho\tfrac{\partial}{\partial u} c_{34}^{(4)}  & \rho \tfrac{\partial}{\partial u} c_{35}^{(4)}  
\end{bmatrix}
\]
as a full rank submatrix. The scheme defined by $\mathcal{C}$ is assumed to have tangent space of dimension at least $2$ at the point $([0:0:0:1], 0)$ and $\rho = \lambda = \epsilon = 1$. Thus $a_{45}^{(4)} = \pm c_{45}^{(4)} \equiv 0 \pmod{u^2}$, as required.
\end{proof}

\begin{proof}[Proof of \Cref{prop:move-stuff}]
    The claim follows from \Cref{lemma:normal-form-5-CC,lemma:divisibility} upon noting that for any $p \in \PP^1$ the image of $G$ under the evaluation map 
    \[
    H^0(\GL(\cE) \times \GL(\cF)) \to ( \GL(\cE) \times \GL(\cF) ) |_p \cong \GL_4(\CC) \times \GL_5(\CC)
    \]contains $U_4(\CC) \times U_5(\CC)$.
\end{proof}

\begin{remark}
  Results similar to \Cref{lemma:divisibility} and \Cref{prop:move-stuff} can be found in several places in the literature, most notably in the arithmetic context of sieving arguments in quintic orders in number fields \cite{B-hcl4}. A new subtlety in our setting is that, in order to lift to $\GL(\cE) \times \GL(\cF)$, it is essential that the claim in \Cref{lemma:divisibility} is obtained after a transformation with $\overbar{g} \in U_4(\CC) \times U_5(\CC)$ (as opposed to $\GL_4(\CC) \times \GL_5(\CC))$. This is because whenever both $e_1 < \dots < e_4$ and $f_1 < \dots < f_5$ the matrices in the image of the evaluation map $H^0(\GL(\cE) \times \GL(\cF)) \to (\GL(\cE) \times \GL(\cF) )|_p$ have zeros above the diagonal. 
\end{remark}

\section{Dimensions of singular loci and the proof of \texorpdfstring{\Cref{thm:deg-5}}{Theorem 1.5}}
\label{sec:dim-counting}

When $\cE$ and $\cF$ satisfy the conditions of \Cref{thm:deg-5}, we prove that a general member of $\cM_{\cE,\cF}$ is smooth by dimension counting. 

\subsection{Normal forms of type $(\K)$}
Mimicking the proof of Bertini's theorem, we first deal with the simpler case where a section $\cA \in H^0 \Big(\bigwedge\nolimits^2 \cF \otimes \cE \otimes \det \cE^\vee \Big)$ can be placed in a normal form of type~$(\K)$.

\begin{proposition}
    \label{prop:case-K}
Let $\cE$ and $\cF$ be vector bundles satisfying the conditions of \Cref{thm:deg-5}. A general section $\cA \in H^0 \Big(\bigwedge\nolimits^2 \cF \otimes \cE \otimes \det \cE^\vee \Big)$ does not have a $G$-translate which is in a normal form of type $(\K)$ for any $1 \leq \K \leq 4$.
\end{proposition}
\begin{proof}
A section $\cA$ has a $G$-translate which is in normal form of type $(\K)$ if and only if there exists a point $p \in \PP^1$ and a point $\boldsymbol{\xi} = (\xi_1, \dots ,\xi_4) \in \cE^{\vee}_p \simeq \CC^4$ such that the evaluation $\cA(\boldsymbol{\xi})\vert_p = 0$. 

Writing $V_j = \oplus_{i=1}^j \cO(-e_i)$, the vector bundle $\cE^{\vee}$ has a natural filtration:
\[
    0 = V_0 \subseteq V_1 \subseteq V_2 \subseteq V_3 \subseteq V_4 = \cE^{\vee}.
\]
For a point $p \in \PP^1$, let $V_{k,p}$ be the restriction of $V_k$ to $p$. We show that for all $k \in \{1,\dots, 4\}$ if $\cA$ is a general section then there do not exist points $p \in \PP^1$ and $\boldsymbol{\xi} = (\xi_1 , \dots ,\xi_4) \in V_{k,p} \setminus V_{k-1,p}$ such that $\cA(\boldsymbol{\xi})\vert_p = 0$.

For each $k \in \{1,\dots, 4\}$ consider the $\CC$-vector space $Z \subseteq H^0 \Big(\bigwedge\nolimits^2 \cF \otimes \cE \otimes \det \cE^\vee \Big)$ consisting of sections with $\cA(\mathbf{x})_{ij} = 0$ for all pairs $(i,j)$ with $i +j + k < 8$. Let $H \subseteq  \bigwedge^2 \CC^5$ be the hyperplane whose elements are $5 \times 5$ alternating matrices whose $(i,j)$-entry is zero if $i +j + k < 8$. Fix a point $p \in \PP^1$ and $\boldsymbol{\xi} = (\xi_1 , \dots, \xi_4) \in V_{k,p} \setminus V_{k-1,p}$. Writing $P = (p,\boldsymbol{\xi})$ there is a natural evaluation map $\ev_P : Z \to H$ given by $\cA \mapsto \cA(\boldsymbol{\xi})|_p $.

Since $\cE$ and $\cF$ satisfy the conditions of \Cref{thm:deg-5}, the evaluation map is surjective (because for every $1 \leq i < j \leq 5$  with $i + j + k \geq 8$, we have $d_{ij}^{(k)} \geq 0$). Therefore, 
\[
    \dim(\ker (\ev_P)) = \dim(Z) - \dim(H).
\]
Moreover, we have:
\[
\dim(H) = 
\begin{cases}
  4 & \text{if } k = 1, \\
  6  & \text{if } k = 2, \\
  8  & \text{if } k = 3, \\
  9 & \text{if } k = 4. \\
\end{cases}
\]
Therefore, 
\begin{equation}
\label{eqn:dim-bound-K}
 \dim\Big (\bigcup_P \ker(\ev_P)\Big) \leq  \dim(Z) - \dim(H) + \dim(\PP(V_k)) = \dim(Z) - \dim(H) + k < \dim (Z). 
\end{equation}
\Cref{eqn:dim-bound-K} shows that a general element of $Z$ does \emph{not} have a $G$-translate in normal form of type $(\K)$.
\end{proof}

\subsection{Normal forms of type $(\I,\J,\K)$}
We now turn to the harder case. 
We first define a pair of varieties
\[
\marksing_{\I\J\K} \subset H^0 \Big(\bigwedge\nolimits^2 \cF \otimes \cE \otimes \det \cE^\vee \Big) \quad \text{and} \quad
\widetilde{\marksing_{\I\J\K}} \subseteq H^0 \Big(\bigwedge\nolimits^2 \cF \otimes \cE \otimes \det \cE^\vee \Big).
\]

\vspace{1em}
\begin{defn}\hfill
\begin{enumerate}
    \item We define $\marksing_{\I\J\K}$ to be the variety whose points are those $\cA \in H^0(\bigwedge^2 \cF \otimes \cE \otimes \det \cE^\vee )$ in normal form of type $(\I,\J,\K)$ at the point $0 \in \PP^1$.
    \item We define $\widetilde{\marksing_{\I\J\K}}$ to be the variety whose points are those $\cA \in H^0(\bigwedge^2 \cF \otimes \cE \otimes \det \cE^\vee)$ in normal form of type $(\I,\J,\K)$ at \emph{some} point $p \in \PP^1$.
\end{enumerate}
\end{defn}

Consider the natural map
\[
    \Phi \colon \widetilde{\marksing_{\I\J\K}} \times G \rightarrow H^0 \Big(\bigwedge\nolimits^2 \cF \otimes \cE \otimes \det \cE^\vee \Big).
\]
In what follows, we compute the codimension of the image of $\Phi$ and show that it is positive when $\cE$ and $\cF$ satisfy the conditions of \Cref{thm:deg-5}. For the remainder of this section we write $\codim G_{\I\J\K}$ for the codimension of $G_{\I\J\K}$ in $G$ and we write $\codim \marksing_{\I\J\K}$ for the codimension of $\marksing_{\I\J\K}$ in the space $H^0(\bigwedge^2 \cF \otimes \cE \otimes \det \cE^\vee)$. We write $\nu = \dim \widetilde{\marksing_{\I\J\K}} - \dim \marksing_{\I\J\K}$.

\begin{defn}
Let $u$ be a uniformizer for the point $0 \in \PP^1$. Define $G_{\I\J\K} \subseteq G$ to be the subgroup consisting of pairs $(g,h) \in G$ such that:
\begin{enumerate}[label=(\roman*)]
    \item $g_{k \K} \equiv 0 \pmod{u}$ for each $k \neq \K$; and 
    \item $h_{i j} \equiv 0 \pmod{u}$ for each $i \in \{1,\dots,5\}\setminus\{\I,\J\}$ and $j \in \{\I,\J\}$. 
\end{enumerate}
\end{defn}

\begin{defn}
We say that $(\I,\J,\K)$ is a triple if we have $1 \leq \I < \J \leq 5$ and $1 \leq \K \leq 4$. Given vector bundles $\cE$ and $\cF$ on $\PP^1$, we say that a triple $(\I,\J,\K)$ is \emph{maximal} if each of the following hold:
\begin{enumerate}[label=(\roman*)]
    \item for all $k > \K$, we have $e_{k} > e_{\K}$;
    \item for all $j > \J$, we have $f_{j} > f_{\J}$; and 
    \item for all $i > \I$ with $i \neq \J$, we have $f_{i} > f_{\I}$.
\end{enumerate}
\end{defn}

The following proposition provides the basis of our dimension counting argument, in order to prove \Cref{thm:deg-5}.

\begin{proposition}
\label{lem:main-counting-lemma}
Fix vector bundles $\cE$ and $\cF$. If for every maximal triple $(\I, \J, \K)$ we have 
\[
\nu + \codim G_{\I\J\K} < \codim \marksing_{\I\J\K}
\]
then a general section $\cA \in H^0\Big(\bigwedge\nolimits^2 \cF \otimes \cE \otimes \det \cE^\vee \Big)$ cannot be put in normal form of type $(\I,\J,\K)$ for any triple $(\I,\J,\K)$. 
\end{proposition}

Before proceeding with the proof of \Cref{lem:main-counting-lemma} we first prove the following lemma.
\begin{lemma}
\label{lem:maximality}
Fix vector bundles $\cE$ and $\cF$ and a triple $(\I,\J,\K)$. Then, there exists a maximal triple $(\I',\J',\K')$ such that $\codim G_{\I\J\K} = \codim G_{\I'\J'\K'}$ and $\codim \marksing_{\I\J\K} = \codim \marksing_{\I'\J'\K'}$.    
\end{lemma}
\begin{proof}
Take $\K'$ to be the largest integer such that $e_{K'} = e_\K$, take $\J'$ to be the largest integer such that $e_{\J'} = e_\J$ and take $\I'$ to be the largest integer not equal to $\J'$ such that $e_{\I'} = e_\I$. Now, the claim is obvious, as the spaces are isomorphic.
\end{proof}

\begin{proof}[Proof of \Cref{lem:main-counting-lemma}]
By \Cref{lem:maximality}, we may assume that $\nu + \codim G_{\I\J\K} < \codim \marksing_{\I\J\K}$ for \emph{all} triples $(\I,\J,\K)$. Letting $(\I,\J,\K)$ be any triple, and adding $\dim \marksing_{\I\J\K}$ to both sides, we obtain
\begin{equation}
\label{eqn:expanded}
\nu + \dim \marksing_{\I\J\K} + \dim G - \dim G_{\I\J\K} < h^0\Big(\bigwedge\nolimits^2 \cF \otimes \cE \otimes \det \cE^\vee \Big).
\end{equation}
Now observe that $G_{\I\J\K}$ maps $\marksing_{\I\J\K}$ to itself. As a result,
\begin{align*}
\dim \Phi(\widetilde{\marksing_{\I\J\K}} \times G) &\leq \nu + \dim \Phi({\marksing_{\I\J\K}} \times G) \\
\label{eqn:dim-bound}
&\leq \nu + \marksing_{\I\J\K} + \dim G - \dim G_{\I\J\K}.
\end{align*}
Combining this with \eqref{eqn:expanded} gives
\begin{equation}
\label{eqn:final-bound}
\dim \Phi(\widetilde{\marksing_{\I\J\K}} \times G) < h^0\Big(\bigwedge\nolimits^2 \cF \otimes \cE \otimes \det \cE^\vee \Big).
\end{equation}
\Cref{eqn:final-bound} implies that for all $(\I,\J,\K)$ the loci $\Phi(\widetilde{\marksing_{\I\J\K}} \times G)$ have positive codimension in the affine space $H^0\Big(\bigwedge\nolimits^2 \cF \otimes \cE \otimes \det \cE^\vee \Big)$. Therefore, a general section $\cA$ cannot be put in normal form of type $(\I,\J,\K)$.
\end{proof}

In order to apply \Cref{lem:main-counting-lemma} we first determine formulae for $\codim G_{\I\J\K}$ and the difference $\nu = \dim \widetilde{\marksing_{\I\J\K}} - \dim \marksing_{\I\J\K}$.

\begin{lemma}
\label{prop:gijk-dim}
Fix vector bundles $\cE$ and $\cF$. If $(\I,\J,\K)$ is a maximal triple, then $\codim G_{\I\J\K} = \I + \J + \K - 4$.
\end{lemma}
\begin{proof}
Note that $\codim G_{\I\J\K}$ is
\[
    \#\{k \in \{1,\dots, 4\} : k \neq \K \text{ and } e_{k} \leq e_K\} + \#\{(i,j) \in \{1,\dots,5\} \times \{\I,\J\} : i \not\in \{\I,\J\} \text{ and } f_{i} \leq f_j\}.
\]
By the maximality assumption, this expression simplifies to 
\[
 (\K-1) + (\I-1) + (\J-2) = \I + \J + \K - 4
\]
as required.
\end{proof}

\begin{lemma}
\label{lem:nu-bound}
We have $\nu \in \{0,1\}$. Moreover $\nu = 0$ if and only if all of the following hold:
\begin{enumerate}[label=(\roman*)]
    \item \label{cond1-marksing} $d^{(\K)}_{\I\J} \leq 1$;
    \item $d^{(k)}_{\I\J} \leq 0$ for $k \neq \K$; 
    \item $d^{(\K)}_{\I j} \leq 0$ for $j \notin \{\I,\J\}$; and 
    \item \label{cond4-marksing}$d^{(\K)}_{i \J} \leq 0$ for $i \notin \{\I,\J\}$.
\end{enumerate}
\end{lemma}
\begin{proof}
Suppose the conditions \ref{cond1-marksing}--\ref{cond4-marksing} are satisfied. Then for any point $\cA \in \widetilde{\marksing_{\I\J\K}}$, we have $a_{\I\J}^{(k)} = 0$ for all $1 \leq k \leq 4$, $a_{Ij}^{(\K)} = 0$ for all $1 \leq j \leq 5$, and $a_{i\J}^{(\K)} = 0$ for all $1 \leq i \leq 5$. Hence, $\cA \in \marksing_{\I\J\K}$, from which $\widetilde{\marksing_{\I\J\K}} = \marksing_{\I\J\K}$.

If at least one of \ref{cond1-marksing}--\ref{cond4-marksing} are not satisfied, then for a general member $\cA \in \widetilde{\marksing_{\I\J\K}}$, at least one of the entries of the following form is nonzero: $a_{\I\J}^{(k)}$, $a_{i\J}^{(\K)}$, and $a_{Ij}^{(\K)}$. Therefore a nonempty open locus of $\widetilde{\marksing_{\I\J\K}}$ is a $\GG_a \cong \mathrm{P}U_2 = \left\{ \mymat{1}{0}{*}{1} \right\}$-bundle over a nonempty open locus of $\marksing_{\I\J\K}$ (via its action on $\PP^1$). This implies that $\dim \widetilde{\marksing_{\I\J\K}} = \dim\marksing_{\I\J\K} + 1$.
\end{proof}

\begin{remark}
There exist $\cE$ and $\cF$ satisfying the conditions of \Cref{thm:deg-5} for which $\nu = 0$. Take for example $(e_1,\dots,e_4) = (1,2,3,4)$ and $(f_1,\dots,f_5) = (2,3,4,5,6)$ and $(\I,\J,\K) = (1,2,1)$. This phenomenon will play an important role in our analysis.
\end{remark}

At this point, it seems plausible to check \Cref{lem:main-counting-lemma} using linear programming. When done with appropriate care, this argument should succeed. However, the direct combinatorial argument is not too complicated (and the relevant subtleties are eminently visible within it), so we present it here.
We first give a useful numerical criterion, \Cref{lem:numerical-criterion}, for checking the hypotheses of \Cref{lem:main-counting-lemma}.

\begin{defn}
\label{def:acc}
Fix $\cE$ and $\cF$ and a triple $(\I,\J,\K)$. We say a triple $(i,j,k)$ is \emph{accessory} to $(\I,\J,\K)$ if $d_{ij}^{(k)} \geq 0$ and at least one of the following holds:
\begin{enumerate}[label=(\arabic*)]
    \item \label{acc1} $\I = i$ and $\J = j$;
    \item $\label{acc2} \K = k$ and $\I = i$ and $j \notin \{\I,\J\}$; or
    \item \label{acc3} $\K = k$ and $\J = j$ and $i \notin \{\I,\J\}$.
\end{enumerate}
\end{defn}

\begin{lemma}
\label{lem:numerical-criterion}
Fix $\cE$ and $\cF$ and a triple $(\I,\J,\K)$. Let $n$ be the number of accessory triples to $(\I,\J,\K)$. Suppose either:
\begin{enumerate}[label=(\roman*)]
    \item \label{num-crit-1} that $n \geq \I+\J+\K-2$;
    \item \label{num-crit-2} that $n \geq \I+\J+\K-3$ and $\nu = 0$; or
    \item \label{num-crit-3} that $n \geq \I+\J+\K-3$ and $d_{\I,\J}^{(\K)} \geq 1$.
\end{enumerate}
Then, $\nu + \codim G_{\I\J\K} < \codim \marksing_{\I\J\K}$.
\end{lemma}
\begin{proof}
To see \ref{num-crit-1} and \ref{num-crit-2} note that for each such triple, $\marksing_{\I\J\K}$ is contained in the locus $ a^{(k)}_{ij} \equiv 0 \pmod{u}$. Therefore, if at least $\I+\J+\K-2$ such triples exist, then $\codim \marksing_{\I\J\K} \geq \I + \J + \K - 2$. Now, \Cref{prop:gijk-dim} shows that $\codim G_{\I\J\K} = \I + \J + \K - 4$, and by \Cref{lem:nu-bound} we have $\nu \leq 1$. Combining these two inequalities gives the desired inequality.

The claim in \ref{num-crit-3} follows by a similar argument, noting that $\marksing_{\I\J\K}$ is contained in the locus where $ a^{(\K)}_{\I\J} \equiv 0 \pmod{u^2}$.
\end{proof}

To prove \Cref{thm:deg-5} it remains to classify the cases which satisfy the hypotheses of \Cref{lem:numerical-criterion}. This is the role of \Cref{lem:satisfying-numerical-criterion,lem:extras}. Note that the hypotheses of \Cref{thm:deg-5} imply that $d_{ij}^{(k)} \geq 0$ whenever $i + j + k \geq 8$. 

\begin{lemma}
\label{lem:satisfying-numerical-criterion}
Fix $\cE$ and $\cF$ satisfying the conditions of \Cref{thm:deg-5} and fix a triple $(\I,\J,\K)$. Then there exist at least $\I+\J+\K-2$ accessory triples to $(\I,\J,\K)$ except when:
\begin{enumerate}[label=(\roman*)]
    \item \label{sat-num-crit-1} $(\I,\J,\K) = (1,2,1)$ and $d_{15}^{(1)}, d_{24}^{(1)}, d_{12}^{(4)} < 0$;
    \item \label{sat-num-crit-2} $(\I,\J,\K) = (4,5,1)$ and $d_{15}^{(1)},d_{24}^{(1)} < 0$; or
    \item $(\I,\J,\K) = (4,5,4)$.
\end{enumerate}
\end{lemma}
\begin{proof}
For each triple $(\I,\J,\K)\notin \{(1,2,1),(4,5,1),(4,5,4)\}$ there exist $\I+\J+\K-2$ triples $(i,j,k)$ satisfying \Cref{def:acc}\ref{acc1}--\ref{acc3} and such that $i + j + k \geq 8$. 

If $(\I,\J,\K) = (1,2,1)$ and at least one of $d_{15}^{(1)}, d_{24}^{(1)}, d_{12}^{(4)}$ are nonnegative then there exist $\I+\J+\K-3$  triples $(i,j,k)$ satisfying \Cref{def:acc}\ref{acc1}--\ref{acc3} and such that $i + j + k \geq 8$. By assumption one of $(1,5,1)$, $(2,4,1)$, or $(1,2,4)$ is also an accessory triple.

Similarly if $(\I,\J,\K) = (4,5,1)$ and at least one of $d_{15}^{(1)},d_{24}^{(1)}$ is nonnegative, there exist $\I+\J+\K-3$  triples $(i,j,k)$ satisfying \Cref{def:acc}\ref{acc1}--\ref{acc3} and such that $i + j + k \geq 8$. By assumption one of $(1,5,1)$ or $(2,4,1)$ is also an accessory triple.
\end{proof}

\begin{lemma}
\label{lem:extras}
Suppose $\cE$ and $\cF$ satisfy the hypothesis of \Cref{thm:deg-5}. Suppose that either:
\begin{enumerate}[label=(\roman*)]
    \item \label{extra1} $(\I,\J,\K) = (1,2,1)$ and $d_{15}^{(1)}, d_{24}^{(1)}, d_{12}^{(4)} < 0$;
    \item \label{extra2} $(\I,\J,\K) = (4,5,1)$ and $d_{15}^{(4)}, d_{24}^{(1)} < 0$; or
    \item \label{extra3} $(\I,\J,\K) = (4,5,4)$.
\end{enumerate}
Then there exist at least $\I + \J + \K - 3$ accessory triples to $(\I,\J,\K)$. Moreover, in case \ref{extra1} we have $\nu = 0$ and in both cases \ref{extra2} and \ref{extra3} we have $d_{\I\J}^{(\K)} \geq 1$.
\end{lemma}

\begin{proof}
In case \ref{extra1} note that the conditions of \Cref{thm:deg-5} imply that $d^{(1)}_{25} = 0$. Indeed we have $\I+\J+\K-3 = 1$ and $(2,5,1)$ is an accessory triple. Moreover $\nu = 0$ by \Cref{lem:nu-bound}.

In cases \ref{extra2} and \ref{extra3} the hypotheses of \Cref{thm:deg-5} (and a simple application of the simplex method) show that $d^{(\K)}_{\I \J} \geq 1$. Moreover in each case there exist $\I+\J+\K-3$ triples $(i,j,k)$ satisfying \Cref{def:acc}\ref{acc1}--\ref{acc3} and such that $i + j + k \geq 8$. 
\end{proof}

\begin{remark}
Suppose $\cE$ and $\cF$ satisfy the conditions of \Cref{thm:main-thm}. An easy calculation now shows that $\Phi(\marksing_{454} \times G)$ has codimension $1$ in $H^0\Big(\bigwedge\nolimits^2 \cF \otimes \cE \otimes \det \cE^\vee \Big)$, so the locus of singular sections is always codimension $1$. One might expect that the locus of singular sections is irreducible, but this is not always the case.

Consider the extremal example $(e_1,\dots,e_4) = (1,2,3,4)$ and $(f_1,\dots,f_5) = (2,3,4,5,6)$. The locus of singular sections has at least three components of maximal dimension: $\Phi(\marksing_{121} \times G)$, $\Phi(\marksing_{451} \times G)$, and $\Phi(\marksing_{454} \times G)$.
\end{remark}

\subsection{Finishing the proof of \Cref{thm:deg-5}}

\begin{proof}[Proof of \Cref{thm:deg-5}]
Combining \cite[Section~5]{CL} and \Cref{prop:cond2} shows that if $\cE$ and $\cF$ do not satisfy the hypotheses of \Cref{thm:deg-5} then $C$ is not a smooth irreducible curve.

For the converse, fix $\cE$ and $\cF$ satisfying the hypothesis of \Cref{thm:deg-5}. Combining \Cref{prop:move-stuff,prop:case-K,lem:main-counting-lemma}, to show smoothness it suffices to check that $\nu + \codim G_{\I\J\K} < \codim \marksing_{\I\J\K}$ for all maximal triples $(\I,\J,\K)$. 
But for each maximal triple combining \Cref{lem:satisfying-numerical-criterion,lem:extras} with \Cref{lem:numerical-criterion} shows that we have $\nu + \codim G_{\I\J\K} < \codim \marksing_{\I\J\K}$, as required. Hence, $C$ is smooth. 

To see that $C$ is irreducible, observe that the conditions of \Cref{thm:deg-5} imply that $e_1 \geq 1$, and hence $h^0(C,\cO_C) = h^0(\PP^1,\pi_* \cO_C) = 1$.
\end{proof}

\section{Deducing \texorpdfstring{\Cref{thm:main-thm} from \Cref{thm:deg-5}}{Theorem 1.4 from Theorem 1.5}}
\label{sec:finishing-proof}

We now deduce \Cref{thm:main-thm} from \Cref{thm:deg-5}.  Recall that we say a region $R \subseteq \RR^d$ is a \emph{rational polyhedral cone} if it is the solution set of finitely many linear inequalities $a_1 x_1 + \dots + a_d x_d \leq 0$ with rational coefficients. Further we say $x \in R \cap \ZZ^d$ is \emph{irreducible} if it cannot be written as $z = x+y$ for nonzero elements $x,y \in R \cap \ZZ^d$. 

\begin{lemma}[Gordan's lemma, {\cite[Proposition~1.2.17]{cls:toric}}]
\label{lem:gordon}
If $R$ is a rational polyhedral cone, then the monoid $R \cap \ZZ^d$ is generated by the set of \emph{irreducible} elements.
\end{lemma}

\begin{proof}[Proof of \Cref{thm:main-thm}]

Let $P \subseteq \RR^4$ be the polytope cut out by the inequalities 
\[
e_1 \leq \dots \leq e_4 \quad \text{and} \quad e_{i + j} \leq e_i + e_j
\]
 for all $1 \leq i,j \leq i + j \leq 4$. Let $Q \subseteq \RR^4 \times \RR^5$ be the region cut out by the inequalities  
 \[
e_1 \leq \dots \leq e_4, \quad f_1 \leq \dots \leq f_5 \quad \text{and} \quad 2\sum_{i=1}^4 e_i = \sum_{j=1}^5 f_j
\]
and the conditions in \Cref{thm:deg-5}, where $g = -4 + \sum_{i=1}^4 e_i$. Consider the projection map
\[
    \pr_{\boldsymbol{e}} \colon Q \rightarrow \RR^4 
\]
given by sending $(e_1,\dots,e_4,f_1,\dots,f_5) \mapsto (e_1,\dots,e_4)$. 

\begin{lemma}
\label{lem:projection}
We have
\[
\pr_{\boldsymbol{e}}(Q \cap \ZZ^9) = P \cap \ZZ^4 .
\]
\end{lemma}
\begin{proof}
A calculation using the computer algebra system \texttt{Magma}~\cite{magma} shows that $\pr_{\boldsymbol{e}}(Q) = P$, so $\pr_{\boldsymbol{e}}(Q \cap \ZZ^9) \subset P \cap \ZZ^4$. To complete the proof, it suffices to show that $P \cap \ZZ^4 \subseteq \pr_{\boldsymbol{e}}(Q \cap \ZZ^9)$.

Let $Q' \subseteq \RR^4 \times \RR^5$ be the region given by 
\[
    Q' = Q \cap \Big \{(\mathbf{e},\mathbf{f}) \in \RR^4 \times \RR^5 : f_2 + f_5 + e_1 = \sum_i e_i \Big\}.
\]
A quick look at the conditions of \Cref{thm:deg-5} will show that $Q'$, unlike $Q'$, is a polyhedron. Because $Q' \subseteq Q$, it suffices to show that $P \cap \ZZ^4 \subseteq \pr_{\boldsymbol{e}}(Q' \cap \ZZ^9)$.

Notice that the map $\pr_{\boldsymbol{e}} \colon Q' \rightarrow P$ gives rise to a homomorphism of monoids $\pr_{\boldsymbol{e}} \colon Q' \cap \ZZ^9 \rightarrow P \cap \ZZ^4$. By \Cref{lem:gordon}, the monoid $P \cap \ZZ^4$ is generated by the set of irreducible elements of $P \cap \ZZ^4$, which in this case are:
\[
    (1,1,1,1), (1,1,1,2), (1,1,2,2), (1,2,2,2), (1,2,2,3), (1,2,3,3), (1,2,3,4).
\]
For each of the generators $\mathbf{e} \in P \cap \ZZ^4$ above, we claim there exists a lift $(\mathbf{e},\mathbf{f}) \in Q' \cap \ZZ^9$. This is readily seen by taking $\mathbf{f}$ to be:
\[
    (1,1,2,2,2), (2,2,2,2,2), (2,2,2,3,3), (2,3,3,3,3), (2,3,3,4,4), (2,3,4,4,5), (2,3,4,5,6)
\]
respectively. 
\end{proof}

Fix $g \geq 0$. Let $\cE$ and $\cF$ be rank $4$ and $5$ vector bundles with $\deg(\cE) = g+4$ and $\deg(\cF) = 2(g+4)$. By \Cref{thm:deg-5}, there exists a smooth irreducible cover with Tschirnhausen bundle $\cE$ and first syzygy bundle $\cF$ if and only if $(e_1,\dots,e_4,f_1,\dots,f_5) \in Q$. By \Cref{lem:projection}, 
\[
 \pr_{\boldsymbol{e}}(Q \cap \ZZ^9) = P \cap \ZZ^4,
\]
and \Cref{thm:main-thm} follows.
\end{proof}

\section{Discussion}
\label{sec:discussion}
\subsection{The story of the ``fin'' in \texorpdfstring{\Cref{thm:deg-5}\ref{cond2}}{Theorem 1.5(ii)}} 
A quick glance at \Cref{thm:deg-5} will show that it can be described, just like \Cref{conj:tschirnhausen}, in terms of polytopes. We do that here. Let $\ove_1,\dots,\ove_4,\ovf_1,\dots,\ovf_5$ be the coordinates of $\RR^4 \times \RR^5 \cong \RR^9$ and let $\ovd_{ij}^{(k)} = \ovf_i + \ovf_j + \ove_k - 1$.

Let $\qpolytope \subset \RR^9$ be the region cut out by the linear conditions

\begin{equation}
  \label{eqn:conds}
  \begin{aligned}
    &\ove_1 \leq \dots \leq \ove_4, \qquad\qquad && \ovf_1 \leq \dots \leq \ovf_5, \\
    &2\sum_{i=1}^4 \ove_i = \sum_{i=1}^5 \ovf_i = 2, \qquad\qquad && \ovd_{ij}^{(k)} \geq 0 \text{  for all  }i + j + k = 8, 
  \end{aligned}
\end{equation}
and the condition that if
\begin{equation}
  \label{eqns:conds2}
  \ovd_{15}^{(1)}, \ovd_{24}^{(1)}, \ovd_{12}^{(4)} < 0 \quad \text{then} \quad \ovd_{25}^{(1)} = 0.
\end{equation}
We define four sub-regions $\qpolytope_1, \qpolytope_2, \qpolytope_3, \qpolytope_4 \subset \qpolytope$ as follows. Let $\qpolytope_1 \subset \qpolytope$ be cut out by the conditions in \eqref{eqn:conds} along with the condition $d_{15}^{(1)} \geq 0$, let $\qpolytope_2$ be cut out by by the conditions in \eqref{eqn:conds} along with $d_{24}^{(1)} \geq 0$, let $\qpolytope_3$ be cut out by the conditions in \eqref{eqn:conds} along with $d_{12}^{(4)} \geq 0$, and finally let $\qpolytope_4$ be cut out by the conditions in \eqref{eqn:conds} along with the constraints
\[
\ovd_{15}^{(1)}, \ovd_{24}^{(1)}, \ovd_{12}^{(4)} < 0 \quad \text{and} \quad \ovd_{25}^{(1)} = 0.
\]

The content of our main theorem, \Cref{thm:deg-5}, is that the scrollar invariants achieved from a quintic cover are precisely those which lie in region $\tfrac{1}{g+4} \qpolytope$. Because $\qpolytope_4$ has dimension $1$ less than the other polytopes, we call it the ``fin''. 

Let $\cH_{5,g}$ be the stack of smooth degree $5$ genus $g$ covers of $\PP^1$ and let $\Bun_{\PP^1}$ be the stack of vector bundles on $\PP^1$. The \emph{Tschirnhausen morphism} 
\[\Tsch \colon \cH_{5,g} \rightarrow \Bun_{\PP^1} \times \Bun_{\PP^1}\]\
defined by $[\pi] \rightarrow (\cE, \cF)$ can be interpreted (via the splitting of vector bundles on $\PP^1)$ as a map 
\[
    \Tsch \colon \cH_{5,g} \rightarrow \ZZ^4 \times \ZZ^5.
\]

The moduli stack of quintic covers of fixed genus is irreducible in characteristic 0~\cite{Fulton-hurwitz}, so it is striking that (asymptotically as $g \rightarrow \infty$) its image under $\Tsch$ has four pieces (and interestingly, the piece $\qpolytope_4$ has dimension $1$ less than that of $\qpolytope_i$ for $i = 1, 2, 3$). 

This is not a new phenomenon; in fact already in the case $d = 4$, there is a fin (see \cite[Section~3]{VV}, this is illustrated in \cite[Figure~1]{VV}), and it corresponds to quartic covers factoring through a quadratic subcover. Moreover one can demonstrate the existence of fins in every composite degree in a similar manner (see \cite[Section~3]{VV}). However \Cref{thm:deg-5} provides the first known case of a fin in \emph{prime} degree. 

\subsubsection{Observation}
Given a smooth curve lying on the fin, the curve lies inside the subvariety in $\PP(\cE)$ cut out by the vanishing of the $2 \times 2$ minors of 
\[
    \begin{bmatrix}
        \cA_{13} && \cA_{14} && \cA_{15} \\
        \cA_{23} && \cA_{24} && \cA_{25}
    \end{bmatrix}
\]
Moreover, the $x_1$ coordinate only appears in $\cA_{25}$, and here, its coefficient is a nonzero complex number.

\subsection{Relationship to measures valued in the Grothendieck ring}

In \cite[Section~8]{VV}, the authors propose several conjectures suggesting that the stratification of $\cH_{d,g}$ by scrollar invariants should be interpreted, in the limit as $g \rightarrow \infty$, as a measure valued in the Grothendieck ring of stacks.

Our contribution to these conjectures is twofold. First, \Cref{thm:deg-5} immediately gives a partial answer to some of these conjectures in the case $d = 5$. Second, our method (minimization) gives extremely strong control of the locus of singular covers in $H^0(\bigwedge \cF \otimes \cE \otimes \det \cE^{\vee})$, and may be useful to prove the existence of the aforementioned measure in the case $d = 5$.

\begin{remark}
  Under an analogy made precise in \cite[Section~8]{VV} the scrollar invariants of a degree $d$ cover $C \to \PP^1$ correspond to the successive minima of the ring of integers of a degree $d$ number field $K/\QQ$. Their conjectures (which we reproduce as \Cref{conj:limiting,conj:groth,conj:rho}) are stated in \cite{VV} to include this arithmetic setting.  We only address the geometric case here, but it would be interesting to extend our work to treat the arithmetic case (possibly by proving an ``Arakelov'' version of \Cref{prop:move-stuff} and replacing the dimension counting in \Cref{sec:dim-counting} by a sieve, in the spirit of e.g., \cite{B-quinticdiscs}).
\end{remark}

We now recall some material from \cite[Section~8]{VV}. Fix a positive integer $d \geq 2$. Let $\mathcal{H}_{d,g}$ denote the stack of smooth, degree $d$, genus $g$ covers of $\PP^1$, and let $\mathcal{H}_{\mathbf{e},g}$ denote the substack with scrollar invariants $\mathbf{e}$. Let $\qpolytope$ be the polytope defined by \eqref{eqn:conds} and \eqref{eqns:conds2} and let $\pr_{\mathbf{e}} : \qpolytope \rightarrow \RR^4$ be the natural projection to the $\ove_i$-coordinates.

\begin{defn}[{\cite[Definition~8.2]{VV}}]
For each open ball $B \subset \RR^{d-1}$ and $g>0$, let 
\[\pi^{\textnormal{geo}}(B,g) = \frac{1}{\dim \mathcal{H}_{d,g}}\dim  \left ( \bigcup_{\mathbf{e} \in \ZZ^{d-1} : \mathbf{e}/(d+g-1) \in B}    \mathcal{H}_{\mathbf{e},g} \right )\]
where $\mathcal{H}_{\mathbf{e}, g}$ is the moduli space of genus degree $d$ genus $g$ covers of $\PP^1$  with scrollar invariants $\mathbf{e}$. By convention, the dimension of a finite disjoint union is the maximum of the dimensions, and the dimension of the empty set is $-1$. 
\end{defn}

In \cite{VV} the authors conjecture that the function $\pi^{\textnormal{geo}}(B,g)$ should behave nicely as $g \rightarrow \infty$.

\begin{conjecture}[{\cite[Conjecture~8.4]{VV}}]
\label{conj:limiting}
For all $B \subset \RR^{d-1}$ the limit $\lim_{g \rightarrow \infty} \pi^{\textnormal{geo}}(B,g)$ exists and depends only on the intersection of $B$ with the hyperplane 
\[
\left\{ (\overline{e}_1, \dots, \overline{e}_{d-1}) : \sum \overline{e}_i = 1 \right\}
\]
and not on $B$. 
\end{conjecture}

When $\lim_{g \rightarrow 0}\pi^{\textnormal{geo}}(B,g)$ exists, denote it by $\pi^{\textnormal{geo}}(B)$. When $d = 5$ \Cref{thm:deg-5} implies that the limit exists, as we record below.

\begin{coro}
    \label{coro:groth-ring}
    \Cref{conj:limiting} holds when $d = 5$. Moreover either:
    \begin{enumerate}[label=(\roman*)]
        \item $B \cap \polytope_5 = \emptyset$ and $\lim_{g \rightarrow \infty} \pi^{\textnormal{geo}}(B,g) = 0$; or
    \item
    $B \cap \polytope_5 \neq \emptyset$ and $\lim_{g \rightarrow \infty} \pi^{\textnormal{geo}}(B,g)$ is equal to
    \begin{equation*}
    \label{eqn:pi}
    \sup \left\{ 1 + \frac{1}{2}\sum_{1 \leq i < j \leq 5, 1 \leq k \leq 4} \max\{0,1-\overline{f}_i - \overline{f}_j - \overline{e}_k\} - \frac{1}{2}\sum_{1 \leq i < j \leq 5}(\overline{f}_j - \overline{f}_i) - \frac{1}{2}\sum_{1 \leq k < \ell \leq 4}(\overline{e}_{\ell} - \overline{e}_k) \right \}
    \end{equation*}
    where the supremum is taken over $(\overline{e}_1,\dots,\overline{e}_4, \overline{f}_1, \dots\overline{f}_5) \in \pr_{\mathbf{e}}^{-1}(B) \cap \qpolytope$.
    \end{enumerate}
\end{coro}
 
More generally, in \cite[Section~8]{VV}, the authors define
$$\pi^{\textnormal{gr}}(B,g) = [\mathcal{H}_{d,g}]^{-1}\sum_{{\mathbf{e}} \in \ZZ^{d-1} : {\mathbf{e}}/(d+g-1) \in B}    [\mathcal{H}_{{\mathbf{e}},g}],$$
to take values in the Grothendieck ring of stacks with  $[\mathcal{H}_{d,g}]$ inverted, and conjecture that:

\begin{conjecture}[{\cite[Conjecture~8.6]{VV}}]
\label{conj:groth}
For all $B$ the limit $\lim_{g \rightarrow \infty} \pi^{\textnormal{gr}}(B,g)$ converges to a measure $\pi^{\textnormal{gr}}$ valued in the Grothendieck ring of stacks.  
\end{conjecture}

In the case $d = 5$, inverting $[\mathcal{H}_{d,g}]$ is harmless because $\mathcal{H}_{d,g}$ is the union of (open substacks of) quotients of affine spaces by groups. Proving \Cref{conj:groth} in this case amounts to understanding the singular loci in $H^0(\bigwedge^2 \cF \otimes \cE \otimes \det(\cE)^{\vee})$. Our main technical lemma, \Cref{prop:move-stuff}, gives a concrete description of the singular locus. Proving \Cref{conj:groth} in the case $d = 5$ is outside the scope of the paper, but we note that minimization is a promising direction for this.

Let $\pi^{\textnormal{gr}}(B) = \lim_{g \rightarrow 0}\pi^{\textnormal{gr}}(B,g)$ whenever the limit exists. Finally, for a point $x \in \RR^{d-1}$ define
\[
    \rho^{\textnormal{geo}}(x) = \lim_{i \rightarrow \infty}\pi^{\textnormal{geo}}(B_i(x))
\]
where $\{B_i(x)\}$ is a set of open balls all containing $x$ and with radius converging to $0$. We call $\rho^{\textnormal{geo}}$ a \emph{density function}. Define $\rho^{\textnormal{gr}}(x)$ analogously.

\begin{conjecture}[{\cite[Conjecture~8.11]{VV}}]
\label{conj:rho}
The quantities $\rho^{\textnormal{geo}}(x)$ and $\rho^{\textnormal{gr}}(x)$ exist and are independent of the choice of $\{B_i(x)\}$. They satisfy the relation $\rho^{\textnormal{geo}} = \dim \circ \rho^{\textnormal{gr}}$. Moreover, they are supported on a finite union of rational polytopes, are piecewise linear, and are continuous on their support.
\end{conjecture}

When $d = 5$, \Cref{coro:groth-ring} immediately yields the following corollary. 

\begin{coro}
When $d = 5$, the part of \Cref{conj:rho} pertaining to $\rho^{\textnormal{geo}}$ is true. More precisely, $\rho^{\textnormal{geo}}$ is defined everywhere, is supported on $\polytope_5$, and is given on its support by
\[
 \sup \left\{ 1 + \frac{1}{2}\sum_{1 \leq i < j \leq 5, 1 \leq k \leq 4} \max\{0,1-\overline{f}_i - \overline{f}_j - \overline{e}_k\} - \frac{1}{2}\sum_{1 \leq i < j \leq 5}(\overline{f}_j - \overline{f}_i) - \frac{1}{2}\sum_{1 \leq k < \ell \leq 4}(\overline{e}_{\ell} - \overline{e}_k) \right \}
\]
where the supremum is taken over $(\overline{e}_1,\dots,\overline{e}_4,\overline{f}_1,\dots,\overline{f}_5) \in \pr_{\mathbf{e}}^{-1}(\overline{e}_1,\dots,\overline{e}_4) \cap \qpolytope$. 
\end{coro}

\begingroup
\hypersetup{
  urlcolor=mybiburlcolor
}
\providecommand{\bysame}{\leavevmode\hbox to3em{\hrulefill}\thinspace}
\providecommand{\MR}{\relax\ifhmode\unskip\space\fi MR }
\providecommand{\MRhref}[2]{%
  \href{http://www.ams.org/mathscinet-getitem?mr=#1}{#2}
}
\providecommand{\bibtitleref}[2]{%
  \hypersetup{urlbordercolor=0.8 1 1}%
  \href{#1}{#2}%
  \hypersetup{urlbordercolor=cyan}%
}
\providecommand{\href}[2]{#2}

\endgroup

\end{document}